\documentclass{amsart}

\usepackage{amsthm, amsfonts, amssymb, amsmath, latexsym, enumerate, times}
\usepackage{amscd}
\usepackage{xypic}  
\usepackage{amssymb}
\usepackage{amsthm}
\usepackage{epsfig}
\usepackage{paralist}
\usepackage{lmodern}
\usepackage[T1]{fontenc}
\usepackage[utf8]{inputenc}
\usepackage{mathrsfs}
\usepackage{array}
\usepackage{comment}
\usepackage{paralist}
\usepackage{color}

\newtheorem{theorem}{Theorem}[section]

\newtheorem{proposition}[theorem]{Proposition}

\theoremstyle{definition}

\newcommand{\p}{{\mathbb P}}

\newcommand{\rk}{\operatorname{rk}}

\def\geq{\geqslant}
\def\leq{\leqslant}

\begin{document}
 
\title[Enumeration of Terracini schemes]{Enumeration of Terracini schemes}

\author{Ciro Ciliberto}
\address{Dipartimento di Matematica, Universit\`a di Roma Tor Vergata, Via O. Raimondo 00173, Roma, Italia}
\email{cilibert@mat.uniroma2.it}

 
 
\maketitle

\medskip

 \begin{abstract} In this note we outline a way of computing the expected lenght of the Terracini scheme of a curve, when this scheme is expected to be finite and we give a closed formula for curves in $\p^4$. We also discuss the widely open case of varieties of higher dimension. 
 \end{abstract} 
 
 \section*{Introduction}
 
 Let $X\subset \p^r$ be an irreducible, projective variety, that we assume to be smooth and non--degenerate. Let $p_1,\ldots, p_h$ be $h\geq 2$ general points of $X$. 
 The famous \emph{Terracini's lemma} says that the span of the union of the tangent spaces to $X$ at  $p_1,\ldots, p_h$ is the tangent space to the $(h-1)$--th secant variety ${\rm Sec}_{h-1}(X)$ to $X$ at a general point $x$ of the linear space $\langle p_1,\ldots, p_h\rangle$, that is also a general point of ${\rm Sec}_{h-1}(X)$. The expected dimension of ${\rm Sec}_{h-1}(X)$, and therefore of $T_{{\rm Sec}_r(X),x}$, is $\min\{ r, h(n + 1) - 1 \}$, and the dimension $\sigma_{h-1}(X)$ of ${\rm Sec}_{h-1}(X)$ is bounded above by the expected dimension. If  $\sigma_{h-1}(X)<\min\{ r, h(n + 1) - 1 \}$,  $X$ is said to be \emph{$(h-1)$--defective}. Curves are never defective, whereas in higher dimension there are a lot of defective varieties.
   
In \cite {BC} the Authors introduced the concept of \emph{Terracini loci} of $X$. According to \cite {BC}, the \emph{$h$--th Terracini locus} of $X$ is the subset of $X_h$, the Hilbert scheme of 0--dimensional schemes of length $h$ in $X$, of $h$--tuples of distinct points $p_1,\ldots, p_h$ of $X$ such that the span of the union of the tangent spaces to $X$ at  $p_1,\ldots, p_h$ is strictly smaller than $\sigma_{h-1}(X)$. In particular, if $X$ is not $(h-1)$--defective, the $h$--th Terracini locus of $X$ is the set of distinct $h$--tuples of points $p_1,\ldots, p_h$ of $X$ such that
$$
\dim \Big (\langle \bigcup _{i=1}^h T_{X,p_i}  \rangle   \Big)<\min\{ r, h(n + 1) - 1 \}.
$$
One can actually  give a more formal definition of \emph{Terracini scheme} in $X_h$ (see \S \ref {sec:curves} for the case of curves).  

In this paper we look at the case of non--$(h-1)$--defective varieties $X$ for which the $h$--th Terracini scheme  is expected to be finite and we consider the problem of computing the expected length this scheme.  In \S \ref {sec:curves}  we outline a way of solving this problem in the case of curves, and we provide a closed formula for the aforementioned length in the case of curves in $\p^4$.  In \S \ref {sec:problem} we briefly discuss the case of higher dimensional varieties, which is widely open.
\medskip 

{\bf Acknowledgement}: This problem arose in the course of a problem session during the Workshop ``Geometry of secants'' at the IMPAM of Warsaw in October 2021. The author thanks the organizers and the participants for the very stimulating and nice atmosphere. The author also thanks Paolo Aluffi for interesting discussions about the topic of this note. 

The author is a member of GNSAGA of the Istituto Nazionale di Alta Matematica ``F. Severi''.  \medskip

 \section{The curve case}\label{sec:curves}  Let $n\geq 2$ be an integer and let $C\subset \p^{3n-2}$ be a smooth, irreducible, projective, non--degenerate curve of degree $d$ and genus $g$. We will denote by $L$ the hyperplane class bundle on $C$ and by $V\subset H^0(C,L)$ the $(3n-1)$--dimensional vector space corresponding to the embedding of $C$ in $\p^4$.

Let $C_n$ be the $n$--fold symmetric product of $C$. Consider $C\times C_n$ with the two projections $\pi_1, \pi_2$ to the first and second factor. Let $D\subset C\times C_n$ be the universal divisor. We define a rank $2n$ vector bundle $\mathcal E_C$ on $C_n$ as follows
 $$
 \mathcal E_C=(\pi_2)_*(\mathcal O_{2D}\otimes \pi_1^*(L))
 $$
 and we can look at the evaluation  map
 $$
 \varphi: V\otimes \mathcal O_{C_n}\longrightarrow \mathcal E_C.
 $$
 We can consider the subscheme ${\rm Terr}(C)$ of $C_n$ of points $E\in C_n$ where $\rk(\varphi)<2n$. The scheme ${\rm Terr}(C)$ is called the \emph{Terracini scheme} of $C$. A divisor $E\in C_n$ belongs to ${\rm Terr}(C)$ if and only if the subspace $V(-2E)\subset V$ of sections $s\in V$ such that ${\rm div}(s)-2E$ is effective, has dimension at most $2n-1$. 
 
In this setting one may expect that ${\rm Terr}(C)$ is finite. If so its lenght will be denoted by $t(C)$, and,  by Porteous formula (see \cite [p. 86] {ACGH}), one has
\begin{equation}\label{eq:t}
t(C)=\left|
\begin{array}{cccc}
c_1 & \ldots&c_{n-1} &c_n\\
c_0 & \ldots &c_{n-2} &c_{n-1}\\
\ldots & \ldots& \ldots  &\ldots\\
0 & \ldots &0&c_{1}\\
\end{array}
\right|
\end{equation}
where $c_i=c_i(\mathcal E_C)$ for $0\leq i\leq n$. To compute $t(C)$ one has  to compute the Chern classes $c_i=c_i(\mathcal E_C)$. To do this, one can proceed as follows. 

By Grothendieck--Riemann--Roch theorem, one has
$$
{\rm td}(C_n) \cdot {\rm ch}(\mathcal E_C)=(\pi_2)_*\big ({\rm td}(C\times C_n)\cdot {\rm ch}(\mathcal O_{2D}\otimes \pi_1^*(L))\big ).
$$
Eliminating the factor ${\rm td}(C_n)$, we find
\begin{equation}\label{eq:*}
 {\rm ch}(\mathcal E_C)=(\pi_2)_*\big ({\rm td}(C)\cdot {\rm ch}(\mathcal O_{2D}\otimes \pi_1^*(L))\big ).
\end{equation}

From the exact sequence
$$
0\longrightarrow \pi_1^*(L)\otimes \mathcal O_{C\times C_2}(-2D)\longrightarrow 
\pi_1^*(L)\longrightarrow \mathcal O_{2D}\otimes \pi_1^*(L)\longrightarrow 0
$$
we deduce
\begin{equation}\label{eq:?}
{\rm ch}(\mathcal O_{2D}\otimes \pi_1^*(L))={\rm ch}(\pi_1^*(L))-{\rm ch}(\pi_1^*(L)\otimes \mathcal O_{C\times C_2}(-2D)).
\end{equation}

Let us denote by $\eta$ the class of the pull back by $\pi_1$ of a point $p$ of $C$, i.e., the class of $\{p\}\times C_n\cong C_n$. Of course, one has
$$
\eta^2=0.
$$
We denote by $\delta$ the class of $D$ and by $\theta$ the pull back  of the theta divisor in the jacobian $J(C)$ via the Abel--Jacobi map $a: C_n\longrightarrow J(C)$. By abuse of notation, $\theta$ will also denote the pull back on $C\times C_n$ via $\pi_2$ of this class. One has
$$
\delta=n\eta+\gamma+x
$$
where $x$ denotes the class of a divisor of the form $p+C_{n-1}\subset C_n$, 
with $p\in C$ (by abuse of notation, $x$ will also denote the pull back of this class on $C\times C_n$ via $\pi_2$), and $\gamma$ is a suitable class on $C\times C_n$ (see \cite [p. 338]{ACGH}), that is such that
$$
\gamma^2=-2\eta \theta, \quad \gamma^3=\eta\gamma=0.
$$  
One has
$$
{\rm ch}(\pi_1^*(L))= e^{d\eta}, \quad {\rm ch}(\pi_1^*(L)\otimes \mathcal O_{C\times C_2}(-2D))=e^{d\eta-2\delta}=e^{(d-2n)\eta-2\gamma} e^{-2x}
$$
and 
$$
{\rm td}(C)=1+(1-g)\eta. 
$$
Hence, from \eqref {eq:?} we get
$$
{\rm td}(C)\cdot {\rm ch}(\mathcal O_{2D}\otimes \pi_1^*(L))=
1+(d-g+1)\eta-\big (1+(d-g-2n+1)\eta-2\gamma -4\eta\theta\big )e^{-2x}.
$$

By \eqref {eq:*}, and taking into account that $\eta\gamma=0$ implies $(\pi_2)_*(\gamma)=0$, we find
\begin{equation}\label{eq:x}
{\rm ch}(\mathcal E_C)=(d-g+1)-(d-g-2n+1 -4\theta)e^{-2x}
\end{equation}

Expanding $e^{-2x}$, taking into account that $x^n=1$ and using Poincar\'e's formula to compute $x^i\theta^{n-i}$ for $0\leq i\leq n$, one computes ${\rm ch}(\mathcal E_C)$. Now
$$
{\rm ch}(\mathcal E_C)=2n+\sum_{i=1}^n \frac {p_n}{n!}
$$
where $p_1=c_1$ and $p_i$ is determined inductively by \emph{Newton's formula}
$$
p_i-c_1p_{i-1} +c_2p_{i-2}+\cdots+(-1)^{i-1}c_{i-1}p_1+(-1)^{i}ic_i=0
$$
(see \cite [p. 56]{F}). From this, in principle one can recursively compute the Chern classes of $\mathcal E_C$ and plugging into \eqref {eq:t} one finds the expression for $t(C)$. 

The general computation is cumbersome and it is not easy to come up with a closed formula for $t(C)$ for  all $n\geq 2$. This is however possible for low values of $n$. For example, one has:

\begin{proposition}\label{prop:main} For $n=2$, i.e., for curves in $	\mathbb P^4$, one has
 $$
 t(C)=2(d-g-3)(d-g-4)+8g(d-5).
 $$ 
  \end{proposition} 
 
\begin{proof} Formula \eqref {eq:t} now reads
\begin{equation}\label{eq:p}
t(C)=c_1^2(\mathcal E_C)-c_2(\mathcal E_C)
\end{equation} 
so we have to compute only $c_1^2(\mathcal E_C)$ and $c_2(\mathcal E_C)$. 

Formula \eqref {eq:x} in this case reads
$$
{\rm ch}(\mathcal E_C)=4+2(d-g-3)x+4\theta-2(d-g-3)x^2-8\theta x.
$$

From this, taking into account that $x^2=1, \theta x=g$, we get that
$$
\begin{aligned}
&c_1(\mathcal E_C)=2(d-g-3)x+4\theta\cr
&\frac {c_1(\mathcal E_C)^2-2c_2(\mathcal E_C)}2=-2(d-g-3)x^2-8\theta x=-2(d-g-3)-8g.
\end{aligned}
$$
Since $\theta^2=g(g-1)$ we deduce that
$$
\begin{aligned}
&c_1(\mathcal E_C)^2=4(d-g-3)^2+16(d-g-3)g+16g(g-1)\cr
&c_2(\mathcal E_C) =2(d-g-3)(d-g-2)+8(d-3-g)g+8g^2.
\end{aligned}
$$
and, plugging into \eqref {eq:p}, we see that the assertion holds. \end{proof}

\section{Open problems}\label{sec:problem}

One can consider Terracini schemes also for higher dimensional varieties and one can consider enumerative problems as in the curve case we treated above.  Indeed, let $n\geq 2$ be an integer, and let $X\subset \p^{n(2m+1)-2}$ be a smooth, projective, irreducible, non--degenerate, not $(n-1)$--defective variety of dimension $m\geq 2$. Let $L$ be the hyperplane class bundle on $X$ and let $V\subset H^0(C,L)$ be the $(n(2m+1)-1)$--dimensional vector space corresponding to the embedding of $X$ in $\p^{n(2m+1)-2}$.

Let $X_n$ be the Hilbert scheme of $0$--dimensional subschemes of length $n$ of $X$. 
Consider $X\times X_n$ with the two projections $\pi_1, \pi_2$ to the first and second factor. Let $D\subset X\times X_n$ be the universal scheme. We define a rank $n(m+1)$ vector bundle $\mathcal E_X$ on $X_n$ as follows
 $$
 \mathcal E_X=(\pi_2)_*(\mathcal O_{X\times X_n}/\mathcal I^2_{D}\otimes \pi_1^*(L)),
 $$
 where $\mathcal I_{D}$ is the ideal sheaf of $D$ in $X\times X_n$.
 One can look at the evaluation  map
 $$
 \varphi: V\otimes \mathcal O_{X_n}\longrightarrow \mathcal E_X
 $$
 and consider  the subscheme ${\rm Terr}(X)$ of $X_n$ of points $Z\in X_n$ where $\rk(\varphi)<n(m+1)$. The scheme ${\rm Terr}(X)$ is called the \emph{Terracini scheme} of $X$. A scheme $Z\in X_n$ belongs to ${\rm Terr}(X)$ if and only if the subspace $V(-2Z)\subset V$, with  $V(-2Z) = V\cap H^0(X, L\otimes \mathcal I^2_{Z,X})$, has dimension at most $n(m+1)-1$. 
 
In this setting, one may expect that ${\rm Terr}(X)$ is finite  and one may want to compute its expected length. In principle, one can follow the same argument outlined in the curve case, but, since the universal scheme $D\subset X\times X_n$ is no longer a divisor, the computations become soon very difficult, even in the simplest case $m=n=2$ of surfaces in $\p^8$. It is therefore an open problem to find formulas for the length of ${\rm Terr}(X)$ in the above cases.


\begin{thebibliography}{}

\bibitem {ACGH}  E. Arbarello, M. Cornalba, P. A. Griffiths, J. Harris, \emph{Geometry of Algebraic Curves, Vol. I},  Springer Verlag, Grundlehren der mathematischen Wissenschaften, {\bf 267}, 1985. 

\bibitem {BC} E. Ballico, L. Chiantini, \emph{On the Terracini locus of projective varieties}, Milan Journal of Mathematics,  {\bf 89} (2021), 1--17.

\bibitem {F} W. Fulton, \emph{Intersection Theory}, Springer Verlag, Ergebnisse der Math. und ihre Grenzgeb.,  {\bf 2}, 1984.

 
\end{thebibliography}
\end{document}